\documentclass[a4paper,12pt]{article}

\usepackage[T2A]{fontenc}
\usepackage[utf8]{inputenc}
\usepackage[russian]{babel}
\usepackage{amsmath, amssymb, amsthm, eqnarray}
\usepackage{enumerate}
\usepackage{indentfirst}
\usepackage[colorlinks=true,linkcolor=blue,citecolor=blue]{hyperref}
\usepackage{tikz-cd}
\usetikzlibrary{cd}
\usetikzlibrary{matrix, arrows}

\usepackage{tikz}
\usepackage{tikz-3dplot} 
\usetikzlibrary{fadings}
\tikzfading[name=fade out,
            inner color=transparent!0,
            outer color=transparent!100]
\usetikzlibrary{calc,arrows,decorations.pathreplacing,fadings,3d,positioning}

\title{Собственные симметрии трехмерных цепных дробей.
      }
     \date{}

\author{И.\,А.\,Тлюстангелов}

\theoremstyle{definition}
\newtheorem{definition}{Определение}

\newtheorem*{notation*}{Обозначение}
\theoremstyle{remark}

\newtheorem*{remark*}{Замечание}
\theoremstyle{plain}
\newtheorem{theorem}{Теорема}
\newtheorem*{theorem_dir*}{Теорема Дирихле}
\newtheorem{lemma}{Лемма}

\newtheorem{proposition}{Предложение}
\newtheorem{corollary}{Следствие}
\newtheorem*{statement*}{Утверждение}
\newtheorem*{corollary*}{Следствие}

\newtheorem{proof_m*}{Доказательство теоремы 1}

\DeclareMathOperator{\conv}{conv}

\DeclareMathOperator{\id}{id}

\renewcommand{\phi}{\varphi}

\renewcommand{\vec}[1]{\mathbf{#1}}

\newcommand{\R}{\mathbb{R}}
\newcommand{\Z}{\mathbb{Z}}
\newcommand{\Q}{\mathbb{Q}}

\newcommand{\cK}{\mathcal{K}}

\newcommand{\gA}{\mathfrak{A}}
\newcommand{\gR}{\mathfrak{R}}

\newcommand{\Sl}{\textup{SL}}
\newcommand{\Gl}{\textup{GL}}

\newcommand{\trace}{\textup{Tr}}
\newcommand{\cf}{\textup{CF}}

\newcommand{\gal}{\textup{Gal}}

\begin{document}

\maketitle

\begin{abstract}
В данной работе доказывается критерий наличия у алгебраической цепной дроби собственной палиндромической симметрии в размерности $4$. В качестве многомерного обобщения цепных дробей рассматриваются полиэдры Клейна. 

\end{abstract}

\footnotetext[0]{Исследование выполнено за счёт гранта Российского научного фонда №~22-21-00079.}

\section{Введение}\label{intro}

Обыкновенная цепная дробь действительного числа имеет весьма изящную геометрическую интерпретацию, позволяющую перейти от классического случая к многомерному (см. \cite{klein} и, например, \cite{korkina_2dim}, \cite{german_bordeaux}, \cite{karpenkov_book}). Для описания такого обобщения рассмотрим $l_1,\ldots,l_n$ --- одномерные подпространства пространства $\R^n$, линейная оболочка которых совпадает со всем $\R^n$. Гиперпространства, натянутые на всевозможные $(n-1)$-наборы из этих подпространств, разбивают $\R^n$ на $2^n$ симплициальных конусов. Будем обозначать множество этих конусов через
\[ \mathcal{C}(l_1, \ldots, l_n).\]

Симплициальный конус с вершиной в начале координат $\vec{0}$ будем называть \emph{иррациональным}, если линейная оболочка любой его гиперграни не содержит целых точек, кроме начала координат  $\vec{0}$.

\begin{definition}
  Пусть  $C$ --- иррациональный конус, $C \in \mathcal{C}(l_1, \ldots, l_n)$. Выпуклая оболочка $\cK(C) = \conv(C\cap\Z^{n}\setminus\{\vec{0}\} )$ и его граница  $\partial(\cK(C))$ называются соответственно \emph{полиэдром Клейна} и \emph{парусом Клейна}, соответствующими конусу $C$. Объединение же всех $2^n$ парусов
  \[\cf(l_1, \ldots, l_n) = {\underset{C \, \in \, \mathcal{C}(l_1, \ldots, l_n)}{\bigcup}} \partial(\cK(C))\]
  называется \emph{$(n-1)$-мерной цепной дробью}.
\end{definition}

Особенный интерес представляет так называемый алгебраический случай. Напомним, что оператор из $\Gl_{n}(\Z)$ с вещественными собственными значениями, характеристический многочлен которого неприводим над $\Q$, называется \emph{гиперболическим}. Справедливо следующее утверждение о связи гиперболических операторов с алгебраическими числами в случае произвольного $n$ (подробности см., например, в \cite{german_tlyust_2})

\begin{proposition}\label{prop:more_than_pelle_n_dim}
  Числа $1,\alpha_1,\ldots,\alpha_{n-1}$ образуют базис некоторого вполне вещественного расширения $K$ поля $\Q$ тогда и только тогда, когда вектор $(1,\alpha_1,\ldots,\alpha_{n-1})$ является собственным для некоторого гиперболического оператора $A\in\Sl_n(\Z)$.
  При этом вектора $(1,\sigma_i(\alpha_1),\ldots,\sigma_i(\alpha_{n-1}))$, $i=1,\ldots,n$, где $\sigma_1(=\id),\sigma_2,\ldots,\sigma_n$ --- все вложения $K$ в $\R$, образуют собственный базис оператора $A$.
\end{proposition}

В случае $n=2$ предложение \ref{prop:more_than_pelle_n_dim} позволяет геометрически проинтерпретировать классическую теорему Лагранжа о периодичности обыкновенной цепной дроби. Геометрически теорема Лагранжа означает, что последовательность целочисленных длин и углов паруса одномерной цепной дроби $\cf(l_1, l_2)$ периодична тогда и только тогда, когда направления $l_1$ и $l_2$ являются собственными для некоторого $\Sl_{2}(\Z)$ оператора с различными вещественными собственными значениями (см., например \cite{german_tlyust}).

\begin{definition}
  Пусть $l_1,\ldots,l_n$ --- собственные подпространства некоторого гиперболического оператора $A\in\Gl_n(\Z)$. Тогда $(n-1)$-мерная цепная дробь $\cf(l_1,\ldots,l_n)$ называется \emph{алгебраической}. Мы будем также говорить, что эта дробь \emph{ассоциирована} с оператором $A$ и писать $\cf(A)=\cf(l_1,\ldots,l_n)$. Множество всех $(n-1)$-мерных алгебраических цепных дробей будем обозначать $\gA_{n-1}$.
\end{definition}

Будем называть \emph{группой симметрий} алгебраической цепной дроби $\cf(A)=\cf(l_1,\ldots,l_n)$ множество
\[
  \textup{Sym}_{\Z}\big(\cf(A)\big)=
  \Big\{ G\in\Gl_n(\Z) \ \Big|\ G\big(\cf(A)\big)=\cf(A) \Big\}.
\]
Из соображений непрерывности ясно, что для каждого $G\in\textup{Sym}_{\Z}\big(\cf(A)\big)$ однозначно определена перестановка $\sigma_G$, такая что
\begin{equation} \label{eq:repres}
  G(l_{i})=l_{\sigma_G(i)},\quad i=1,\dots,n.
\end{equation}
И обратно, если для $G\in\Gl_n(\Z)$ существует такая перестановка $\sigma_{G}$, что выполняются соотношения \ref{eq:repres}, то $G\in\textup{Sym}_{\Z}\big(\cf(A)\big)$.

Благодаря теореме Дирихле об алгебраических единицах существует изоморфная $\Z^{n-1}$ подгруппа группы $\textup{Sym}_{\Z}\big(\cf(A)\big)$ (см., например, \cite{german_tlyust_2}). Относительно действия этой подгруппы на любом из $2^n$ парусов возникает фундаментальная область, которую можно отождествить с $(n-1)$-мерным тором (см. \cite{korkina_2dim}). Для каждого элемента $G$, принадлежащего этой подгруппе, $\sigma_G=\id$. Однако, в $\textup{Sym}_{\Z}\big(\cf(A)\big)$, вообще говоря, могут существовать такие элементы $G$, для которых $\sigma_G\neq\id$.

\begin{definition}
  Оператор $G\in\textup{Sym}_{\Z}\big(\cf(A)\big)$, такой что $\sigma_G=\id$, будем называть \emph{симметрией Дирихле} дроби $\cf(A)\in\gA_{n-1}$.
\end{definition}

\begin{definition}
  Оператор $G\in\textup{Sym}_{\Z}\big(\cf(A)\big)$, не являющийся симметрией Дирихле, будем называть \emph{палиндромической симметрией} дроби $\cf(A)$. Если множество палиндромических симметрией цепной дроби непусто, то такую цепную дробь будем называть \emph{палиндромичной}.
\end{definition}

\begin{definition}
  Палиндромическая симметрия $G\in\textup{Sym}_{\Z}\big(\cf(A)\big)$ называется \emph{собственной}, если у оператора $G$ существует неподвижная точка на некотором парусе цепной дроби $\cf(A)$. Палиндромическая симметрия $G\in\textup{Sym}_{\Z}\big(\cf(A)\big)$, не являющаяся собственной, называется  \emph{несобственной}.
\end{definition}

В данной работе нас будут интересовать собственные палиндромические симметрии $\cf(A)$ в размерности $n=4$. А именно, мы докажем критерий наличия такого рода симметрий у цепной дроби $\cf(A)$. Для размерностей $n=2$ и $n=3$ аналогичные критерии уже существуют. 

Для $n=2$, то есть для одномерных цепных дробей, палиндромичность напрямую связана с симметричностью периодов обыкновенных цепных дробей квадратичных иррациональностей. Критерий симметричности периода цепной дроби квадратичной иррациональности восходит к результатам Галуа \cite{galois}, Лежандра \cite{legendre}, Перрона \cite{perron_book} и Крайтчика \cite{kraitchik}. В работе \cite{german_tlyust} дано геометрическое доказательство этого критерия. Аналог этого критерия для $n=3$ был получен в работе \cite{german_tlyust_2}. Упомянутые критерии выглядят следующим образом:

\begin{proposition}\label{two_dimension}
 Пусть $\cf(l_1,l_2)\in\gA_1$ и пусть подпространство $l_1$ порождено вектором $(1, \alpha)$. Тогда $\cf(l_1, l_2)$ имеет собственную симметрию в том и только в том случае, если существует такое алгебраическое число $\omega$ степени $2$ со своим сопряжённым $\omega'$, что выполнено хотя бы одно из следующих условий:

  \textup{(а)} $(1, \alpha) \sim(1,\omega):\hskip 6.5mm \trace(\omega)=\omega + \omega' =0$;

  \textup{(б)} $(1, \alpha) \sim(1,\omega):\hskip 6.5mm \trace(\omega)=\omega + \omega' =1$.
  
\end{proposition}

\begin{proposition}\label{three_dimension}
 Пусть $\cf(l_1,l_2,l_3)\in\gA_2$ и пусть подпространство $l_1$ порождено вектором $(1, \alpha, \beta)$. Тогда $\cf(l_1, l_2, l_3)$ имеет собственную симметрию в том и только в том случае, если существует такое алгебраическое число $\omega$ степени $3$ со своими сопряжёнными $\omega'$ и $\omega''$, что выполнено хотя бы одно из следующих условий:

  \textup{(а)} $(1, \alpha, \beta)\sim(1, \omega, \omega'):\hskip 6.5mm \trace(\omega)=\omega + \omega' + \omega'' =0$;

  \textup{(б)} $(1, \alpha, \beta)\sim(1, \omega, \omega'):\hskip 6.5mm \trace(\omega)=\omega + \omega' + \omega'' =1$.

\noindent При выполнении утверждения \textup{(а)} или \textup{(б)} кубическое расширение $\Q(\alpha, \beta)$ будет нормальным.
\end{proposition}

В этих формулировках $\vec v_1\sim\vec v_2$ для векторов из $\R^n$ означает существование такого оператора $X\in\Gl_n(\Z)$ и такого ненулевого $\mu\in\R$, что $X\vec v_1=\mu\vec v_2$.

\section{Формулировка основного результата}\label{results}

Основным результатом данной работы является следующая

 \begin{theorem}\label{theorem_proper}
  Пусть $\cf(l_1,l_2,l_3,l_4)\in\gA_3$ и пусть подпространство $l_1$ порождено вектором $(1, \alpha, \beta, \gamma)$. Пусть $K = \Q(\alpha, \beta, \gamma)$ и $\sigma_1(=\id), \sigma_2, \sigma_3, \sigma_4$ --- все вложения поля $K$ в $\R$ (см. предложение \ref{prop:more_than_pelle_n_dim}). Тогда $\cf(l_1, l_2, l_3, l_4)$ имеет собственную палиндромическую симметрию в том и только в том случае, если (с точностью до перестановки индексов) выполняется
\[ \sigma_3(K) = K, \quad \sigma_4(K) = \sigma_2(K), \quad \sigma_{3}^{2} = \sigma_1 = \id, \quad \sigma_4 = \sigma_2 \sigma_3\]
и существуют такие алгебраические числа $\omega$ и $\psi$ степени $4$, принадлежащие полю $K$, что выполнено хотя бы одно из следующих условий:
  
  \textup{(1)} $(1, \alpha, \beta, \gamma)\sim(1, \omega,  \psi, \omega'):\hskip 14.5mm \psi + \psi'= - (\omega + \omega')$;
  
  \textup{(2)} $(1, \alpha, \beta, \gamma)\sim(1, \omega,  \psi, \omega'):\hskip 14.5mm \psi + \psi'= 1 - (\omega + \omega')$;
  
  \textup{(3)} $(1, \alpha, \beta, \gamma)\sim(1, \omega,  \psi, \omega'):\hskip 14.5mm \psi + \psi'= 2 - (\omega + \omega')$;
  
  \textup{(4)} $(1, \alpha, \beta, \gamma)\sim(1, \omega,  \psi, \frac{\omega + \omega'}{2}):\hskip 10.2mm \psi + \psi'= - (\omega + \omega')$;
  
  \textup{(5)} $(1, \alpha, \beta, \gamma)\sim(1, \omega,  \psi, \frac{\omega + \omega'}{2}):\hskip 10.2mm \psi + \psi'= 2 - (\omega + \omega')$;
  
  \textup{(6)} $(1, \alpha, \beta, \gamma)\sim(1, \omega,  \psi,  \frac{\omega + \omega' + 1}{2}):\hskip 6.5mm \psi + \psi'= - (\omega + \omega')$;
  
  \textup{(7)} $(1, \alpha, \beta, \gamma)\sim(1, \omega,  \psi, \frac{\omega + \omega' + 1}{2}):\hskip 6.5mm \psi + \psi'= 2 - (\omega + \omega')$;
  
  \textup{(8)} $(1, \alpha, \beta, \gamma)\sim(1, \omega,  \psi, \frac{\omega + \omega'}{2}):\hskip 10.2mm \psi + \psi'= 1 - \frac{\omega + \omega'}{2}$;
  
  \textup{(9)} $(1, \alpha, \beta, \gamma)\sim(1, \omega,  \psi, \frac{\omega + \omega'}{2}):\hskip 10.2mm \psi + \psi'= 2 - \frac{\omega + \omega'}{2}$;
  
  \textup{(10)} $(1, \alpha, \beta, \gamma)\sim(1, \omega,  \psi, \frac{\omega' - \omega}{4}):\hskip 8.2mm \psi + \psi'= 2 - \frac{\omega + \omega'}{2}$,

\noindent где $\omega' = \sigma_3(\omega), \psi' = \sigma_3(\psi)$.
  
\end{theorem}

В работе \cite{tlyust} исследуются \emph{циклические} симметрии, то есть такие симметрии $\cf(l_1,l_2,l_3,l_4)\in\gA_3$, которые циклически переставляют направления $l_1,l_2,l_3,l_4$. А именно, в этой работе доказывается, что цепная дробь $\cf(l_1,l_2,l_3,l_4)\in\gA_3$ имеет собственную циклическую симметрию тогда и только тогда, когда $K$ --- циклическое расширение Галуа, $\gal(K/\Q) = \langle \sigma_{2} \rangle$ и выполняется один из пунктов (1) - (7) теоремы \ref{theorem_proper} для $\psi = \sigma_2(\omega)$. Следующее утверждение показывает, что не всякая палиндромичная цепная дробь обладает собственными циклическими симметриями:

\begin{proposition}\label{proper_but_not_cyclic}
Существуют такие вещественные числа $\alpha$, $\beta$, $\gamma$, что подпространство $l_1$ порождено вектором $(1, \alpha, \beta, \gamma)$, вполне вещественное расширение $K=\Q(\alpha,\beta, \gamma)$ поля $\Q$ не является нормальным и $\cf(l_1,l_2,l_3,l_4)\in\gA_{3}$ --- палиндромичная цепная дробь, не обладающая собственными циклическими симметриями.
\end{proposition}

Любопытно, что критерий наличия собственных циклических симметрий не следует непосредственно из теоремы \ref{theorem_proper}. В связи с этим возникает естественный 

\vskip 2.2mm
\textsc{Вопрос.}
  \textit{
 Верно ли, что существует такая палиндромичная алгебраическая цепная дробь $\cf(l_1,l_2,l_3,l_4)$, для которой поле $K$ является циклическим расширением Галуа и у которой не существует собственных циклических палиндромических симметрий?}
 
Оставшаяся часть статьи имеет следующую структуру: в параграфе \ref{permutation_properity_n_4} мы анализируем то, как у собственных симметрий трехмерных цепных дробей устроены собственные подпространства и перестановки из соотношения (\ref{eq:repres}); в параграфе \ref{geom_proper_n_4} мы изучаем геометрию трехмерных цепных дробей, обладающих собственными симметриями; в параграфе \ref{matrix_and_algebraicity} мы устанавливаем связь между определенными классами цепных дробей и матрицами их собственных симметрий; параграф \ref{proof_theorem_proper} посвящен доказательству теоремы \ref{theorem_proper}; наконец, в параграфе \ref{proof_proper_but_not_cyclic} мы доказываем предложение \ref{proper_but_not_cyclic}.

\section{Собственные симметрии и собственные подпространства}\label{permutation_properity_n_4}
 Если задана дробь $\cf(l_1, \ldots, l_n)=\cf(A)\in\gA_{n-1}$, будем считать, что подпространство $l_1$ порождается вектором  $\vec l_1=(1,\alpha_1, \dots, \alpha_{n-1})$ (данное допущение корректно в силу предложения \ref{prop:more_than_pelle_n_dim}). Тогда из предложения \ref{prop:more_than_pelle_n_dim} следует, что числа $1,\alpha_1, \dots, \alpha_{n-1}$ образуют базис поля $K=\Q(\alpha_1, \dots, \alpha_{n-1})$ над $\Q$ и каждое $l_i$ порождается вектором $\vec l_i=(1,\sigma_i(\alpha_1), \dots, \sigma_i(\alpha_{n-1}))$, где $\sigma_1(=\id),\sigma_2, \dots, \sigma_n$ --- все вложения $K$ в $\R$. Заметим, что верна следующая 
\begin{lemma}\label{rational_eigen}
  Пусть $G\in\textup{Sym}_{\Z}\big(\cf(A)\big)$ и $\cf(A)=\cf(l_1, \ldots, l_n)$. Пусть $G \neq \pm I_{n}$ и $G(\vec{l}_1) = \lambda \vec{l}_1$. Тогда  $\lambda \notin \Q$.
  \end{lemma}
\begin{proof}
Предположим, что $\lambda \in \Q$. Поскольку $G \in \Gl_{n}(\Z)$ и $G \neq \pm I_{n}$, то $\textup{rank} (G - \lambda I_{n}) > 0$. Так как $(G - \lambda I_{n})(\vec{l}_1) = \vec{0}$, то какие-то числа из набора $1,\alpha_1,\ldots,\alpha_{n-1}$ выражаются через оставшиеся числа этого набора в виде некоторой линейной комбинации с коэффициентами из $\Q$. В силу предложения \ref{prop:more_than_pelle_n_dim} получаем противоречие.

\end{proof}

Отныне будем считать, что $n=4$, то есть будем рассматривать трехмерные цепные дроби. Напомним также, что для каждого $G\in\textup{Sym}_{\Z}\big(\cf(A)\big)$ соотношением \ref{eq:repres} определена перестановка $\sigma_G$. 

\begin{lemma}\label{ord_3}
  Пусть $G$ --- палиндромическая симметрия $\cf(l_1,l_2,l_3, l_4)\in\gA_3$, ассоциированной с (гиперболическим) оператором $A$.
  Тогда существует такая нумерация подпространств $l_1, l_2, l_3, l_4$, что $\sigma_{G} = (1, 2)(3, 4)$ или $\sigma_{G} = (1, 2, 3, 4)$.
\end{lemma}

\begin{proof}
  Случай $\sigma_{G} = \textup{id}$ невозможен в силу того, что оператор $G$ не является симметрией Дирихле $\cf(A)$.

 Предположим существует такая нумерация подпространств $l_1, l_2, l_3, l_4$, что $\sigma_{G} = (1)(2, 3, 4)$. Таким образом существуют такие вещественные числа $\mu_{1}$, $\mu_{2}$, $\mu_{3}$, $\mu_{4}$, что матрица оператора $G$ в базисе $\vec{l}_{1}, \vec{l}_{2}, \vec{l}_{3}, \vec{l}_{4}$ имеет вид
 \[
   \begin{pmatrix}
     \mu_{1} & 0 & 0 & 0\\
     0 & 0 & 0 & \mu_{2}\\
      0 & \mu_{3} & 0 & 0\\
       0 & 0 &  \mu_{4} & 0
   \end{pmatrix}.
 \]
  Тогда характеристический многочлен оператора $G$ имеет вид
  \[\chi_{G}(x) = (x - \mu_{1})(x^3 - \mu_{2}\mu_{3}\mu_{4}) = x^4 - \mu_{1} x^3 - \mu_{2}\mu_{3}\mu_{4} x \pm 1 \in \Z[x].\]
  Следовательно, $\mu_{1}$ --- целое число, и при этом $\mu_{1}$ --- корень уравнения $\chi_{G}(x) = 0$, то есть $\mu_{1} = \pm 1$. Стало быть, $l_1$ --- собственное подпространство оператора $G$, соответствующее собственному значению $\mu_{1} = \pm 1$. То есть $l_1$ рационально, что противоречит гиперболичности оператора $A$.
  
 Предположим существует такая нумерация подпространств $l_1, l_2, l_3, l_4$, что $\sigma_{G} = (1) (2) (3, 4)$. Таким образом существуют такие вещественные числа $\mu_{1}$, $\mu_{2}$, $\mu_{3}$, $\mu_{4}$, что матрица оператора $G$ в базисе $\vec{l}_{1}, \vec{l}_{2}, \vec{l}_{3}, \vec{l}_{4}$ имеет вид
 \[
   \begin{pmatrix}
     \mu_{1} & 0 & 0 & 0\\
     0 & \mu_{2} & 0 & 0\\
      0 & 0 & 0 & \mu_{3}\\
       0 & 0 &  \mu_{4} & 0
   \end{pmatrix}.
 \]
 Тогда характеристический многочлен оператора $G$, коэффициенты которого целочисленны, имеет вид
  \[(x - \mu_{1})(x - \mu_{2})(x^2 - \mu_{3}\mu_{4}) =\]
  \[ = x^4 - (\mu_{1} + \mu_{2})x^3 + (\mu_{1}\mu_{2} - \mu_{3}\mu_{4})x^2  + (\mu_{1} + \mu_{2}) \mu_{3}\mu_{4}x - \mu_{1}\mu_{2}\mu_{3}\mu_{4}.\]
 Так как $\mu_{1} + \mu_{2} \in \Z$, то $\mu_{3}\mu_{4} \in \Q$. Тогда существуют такие взаимно-простые целые числа $p \ge 1$ и $q \ge 1$, что $|\mu_{3}\mu_{4}| = \frac{p}{q}$, $|\mu_{1}\mu_{2}| = \frac{q}{p}$, а значит
 \[ |\mu_{1}\mu_{2} - \mu_{3}\mu_{4}| = \frac{\pm p^2 \pm q^2}{pq}.\]
 
 Итак, $p^2$ делится на $q$ и $q^2$ делится на $p$, то есть $p=q=1$ и $\mu_{3}\mu_{4} = \pm 1$. Таким образом матрица оператора $G^2$ в базисе $\vec{l}_{1}, \vec{l}_{2}, \vec{l}_{3}, \vec{l}_{4}$ имеет вид
 \[
   \begin{pmatrix}
     \mu_{1}^2 & 0 & 0 & 0\\
     0 & \mu_{2}^2 & 0 & 0\\
      0 & 0 & \pm 1 & 0\\
       0 & 0 &  0 & \pm 1
   \end{pmatrix}.
 \]
Из леммы \ref{rational_eigen} следует, что $G^2 = I_{4}$, то есть $\mu_{1} = \pm 1$. Вновь применяя лемму \ref{rational_eigen} получаем, что $G = \pm I_{4}$, чего не может быть. 
 
   Таким образом существует такая нумерация подпространств $l_1, l_2, l_3, l_4$, что $\sigma_{G} = (1, 2)(3, 4)$ или $\sigma_{G} = (1, 2, 3, 4)$.

  \end{proof}
 
 \begin{corollary}\label{all_to_2_2}
Пусть $G$ --- палиндромическия симметрия $\cf(l_1,l_2,l_3, l_4)\in\gA_3$. Пусть $G' = G^{2}$, если $\textup{ord}(\sigma_{G}) = 4$, и $G' = G$, если $\textup{ord}(\sigma_{G}) = 2$. Тогда $G'$ --- палиндромическия симметрия $\cf(l_1,l_2,l_3, l_4)$ и $\textup{ord}(\sigma_{G'}) = 2$.
\end{corollary}

Пусть $G$ --- палиндромическия симметрия $\cf(l_1,l_2,l_3, l_4)\in\gA_3$.  Изменив при необходимости нумерацию подпространств $l_1, l_2, l_3, l_4$, в силу леммы \ref{ord_3} можно рассмотреть такие вещественные числа $\mu_{1}$, $\mu_{2}$, $\mu_{3}$, $\mu_{4}$, что матрица оператора $G$ в базисе $\vec{l}_{1}, \vec{l}_{2}, \vec{l}_{3}, \vec{l}_{4}$ имеет вид
  \begin{equation}\label{matrix_ord_4_0}
   \begin{pmatrix}
     0 & 0 & 0 &  \mu_{1}\\
     \mu_{2} & 0 & 0 & 0\\
      0 & \mu_{3} & 0 & 0\\
       0 & 0 &  \mu_{4} & 0
   \end{pmatrix}
 \end{equation}
или вид 
 \begin{equation}\label{matrix_ord_2_0}
   \begin{pmatrix}
     0 & 0 & \mu_{1} &  0\\
     0 & 0 & 0 &  \mu_{2}\\
     \mu_{3} & 0 & 0 & 0\\
     0 & \mu_{4} &  0 & 0
   \end{pmatrix}.
 \end{equation}

Пусть $G$ --- палиндромическая симметрия $\cf(l_1,l_2,l_3, l_4)\in\gA_3$ и матрица оператора $G$ в базисе $\vec{l}_{1}, \vec{l}_{2}, \vec{l}_{3}, \vec{l}_{4}$ имеет вид \ref{matrix_ord_4_0}. В работе \cite{tlyust} доказывается, что $G$ является собственной симметрией $\cf(l_1,l_2, l_3, l_4)$ тогда и только тогда, когда $\mu_{1}\mu_{2}\mu_{3}\mu_{4} = 1$ (см. следствие 1 в \cite{tlyust}). Для палиндромических симметрий вида \ref{matrix_ord_2_0} справедливо аналогичное утверждение:

\begin{lemma}\label{prod_lemma}
  Пусть $G$ --- палиндромическая симметрия $\cf(l_1,l_2,l_3, l_4)\in\gA_3$ и матрица оператора $G$ в базисе $\vec{l}_{1}, \vec{l}_{2}, \vec{l}_{3}, \vec{l}_{4}$ имеет вид \ref{matrix_ord_2_0}. Тогда $G$ является собственной симметрией дроби $\cf(l_1,l_2, l_3, l_4)$ в том и только том случае, если $\mu_{1}\mu_{3} = \mu_{2}\mu_{4} = 1$.
\end{lemma}
  
\begin{proof}
Пусть $G$ является собственной симметрией цепной дроби $\cf(l_1,l_2, l_3, l_4)$. Тогда существуют такие числа $\varepsilon_{1}$, $\varepsilon_{2}$, $\varepsilon_{3}$, $\varepsilon_{4}$ из множества $\{-1, 1\}$, что
\[G(\varepsilon_{1}\vec{l}_1, \varepsilon_{2}\vec{l}_2, \varepsilon_{3}\vec{l}_3, \varepsilon_{4}\vec{l}_4) =  \big(\mu_3\varepsilon_{1}\vec l_3,\mu_4\varepsilon_{2}\vec l_4,\mu_1\varepsilon_{3}\vec l_1,\mu_2\varepsilon_{4}\vec l_2\big),\]
и выполняются неравенства
\[\mu_1\frac{\varepsilon_{3}}{\varepsilon_{1}} > 0, \, \, \mu_2\frac{\varepsilon_{4}}{\varepsilon_{2}} > 0, \, \, \mu_3\frac{\varepsilon_{1}}{\varepsilon_{3}} > 0, \, \, \mu_4\frac{\varepsilon_{2}}{\varepsilon_{4}} > 0.\]
Стало быть, $\mu_{1}\mu_{3} > 0$ и $\mu_{2}\mu_{4} > 0$. Так как у оператора $G$ существует неподвижная точка на некотором парусе, то у оператора $G$ существует одномерное собственное подпространство, соответствующее собственному значению $1$. Теперь, поскольку характеристический многочлен оператора $G$ имеет вид $(x^2-\mu_{1}\mu_{3})(x^2-\mu_{2}\mu_{4})$, то $\mu_{1}\mu_{3} = 1$ или $\mu_{2}\mu_{4} = 1$. Тогда $\mu_{1}\mu_{3} = \mu_{2}\mu_{4} = 1$.

Если $\mu_{1}\mu_{3} = \mu_{2}\mu_{4} = 1$, то, опять же, характеристический многочлен оператора $G$ имеет вид $x^{4} - 2x^{2} + 1$. Стало быть у оператора $G$ существует целочисленный собственный вектор, соответствующий собственному значению $1$. Этот вектор лежит внутри некоторого конуса $C \in \mathcal{C}(l_1, l_2, l_3, l_4)$, поскольку цепная дробь $\cf(l_1,l_2,l_3, l_4)$ является алгебраической. 

\end{proof}

 \begin{corollary}\label{property_ord_eq}
Пусть $G$ --- палиндромическия симметрия $\cf(l_1,l_2,l_3, l_4)\in\gA_3$. Тогда $G$ является собственной симметрией в том и только том случае, если $G'$ (см. следствие \ref{all_to_2_2}) является собственной симметрией. 
\end{corollary}

\begin{proof}
Если $G$ является собственной симметрией $\cf(l_1,l_2,l_3, l_4)$, то, очевидно, оператор $G'$ также является собственной симметрией цепной дроби $\cf(l_1,l_2,l_3, l_4)$. 

Обратно, предположим $G'$ --- собственная симметрия $\cf(l_1,l_2,l_3, l_4)$. Если $\textup{ord}(\sigma_{G}) = 4$, то, изменив при необходимости нумерацию подпространств $l_1, l_2, l_3, l_4$, можно считать, что матрица оператора $G$ в базисе $\vec{l}_{1}, \vec{l}_{2}, \vec{l}_{3}, \vec{l}_{4}$ имеет вид \ref{matrix_ord_2_0}. Тогда матрица оператора $G' = G^{2}$ в базисе $\vec{l}_{1}, \vec{l}_{2}, \vec{l}_{3}, \vec{l}_{4}$ имеет вид

\[
   \begin{pmatrix}
     0 & 0 & \mu_{1}\mu_{4} &  0\\
     0 & 0 & 0 &  \mu_{2}\mu_{1}\\
     \mu_{3}\mu_{2} & 0 & 0 & 0\\
     0 & \mu_{4}\mu_{3} &  0 & 0
   \end{pmatrix}.
\]
 Стало быть, в силу леммы \ref{prod_lemma}, $\mu_{1}\mu_{2}\mu_{3}\mu_{4} = 1$, а значит $G$ --- собственная симметрия цепной дроби $\cf(l_1,l_2, l_3, l_4)$ (см. следствие 1 в \cite{tlyust}).

\end{proof}

  \begin{lemma}\label{rational_subspace_2_2}
  Пусть $G$ --- собственная симметрия $\cf(l_1,l_2,l_3, l_4)\in\gA_3$ и $\textup{ord}({\sigma_{G}}) = 2$. Тогда существуют такие одномерные рациональные подпространства $l^{1}_{+}$, $l^{2}_{+}$, $l^{1}_{-}$ и $l^{2}_{-}$, что $G l^{1}_{+} = l^{1}_{+}$, $G l^{2}_{+} = l^{2}_{+}$,  $G l^{1}_{-} = l^{1}_{-}$,  $G l^{2}_{-} = l^{2}_{-}$ и $l^{1}_{+} + l^{2}_{+} + l^{1}_{-} + l^{2}_{-}= \R^{4}$. При этом подпространства $l^{1}_{+}$ и $l^{2}_{+}$ соответствуют собственному значению $1$, а подпространства $l^{1}_{-}$ и $l^{2}_{-}$ соответствуют собственному значению $-1$.
  \end{lemma}
 
  \begin{proof}
   Изменив при необходимости нумерацию подпространств $l_1, l_2, l_3, l_4$, в силу леммы \ref{prod_lemma} можно считать, что существуют такие вещественные числа $\mu_{1}$ и $\mu_{2}$, что  матрица оператора $G$ в базисе $\vec{l}_{1}, \vec{l}_{2}, \vec{l}_{3}, \vec{l}_{4}$ имеет вид
\[
   \begin{pmatrix}
     0 & 0 & \mu_{1} &  0\\
     0 & 0 & 0 &  \mu_{2}\\
     \frac{1}{\mu_{1}} & 0 & 0 & 0\\
     0 & \frac{1}{\mu_{2}} &  0 & 0
   \end{pmatrix}.
 \]

Так как $\chi_{G}(x) = (x - 1)^2(x + 1)^2$, то у оператора $G$ есть двумерное инвариантное подпространство $L_{+}$, соответствующее собственному значению $1$ и двумерное инвариантное подпространство $L_{-}$, соответствующее собственному значению $-1$. Покажем рациональность подпространств $L_{+}$ и $L_{-}$, из чего будет следовать утверждение леммы.

Поскольку подпространство $L_{+}$ совпадает с решением системы линейных уравнений
\[(G-I_4)\vec{x}^\top = \vec{0},\]
то фундаментальная система решений данной системы линейных уравнений имеет размерность 2. Рассмотрев в качестве значений свободных переменных наборы $(0, 1)$ и $(1, 0)$, мы определим два линейно-независимых рациональных решения данной системы, из чего следует рациональность  $L_{+}$. Рациональность подпространства $L_{-}$ доказывается аналогичным способом.

\end{proof}

\section{Геометрия собственных симметрий}\label{geom_proper_n_4}
  
 \begin{lemma}\label{main_lem_2}
   Пусть $G$ --- собственная симметрия дроби $\cf(l_1,l_2,l_3, l_4)\in\gA_3$. Пусть $F=G'$ (см. следствие \ref{all_to_2_2}) --- собственная симметрия $\cf(l_1,l_2,l_3, l_4)$ (см. следствие \ref{property_ord_eq}). Тогда существуют $\vec{z}_1$, $\vec{z}_2$, $\vec{z}_3$, $\vec{z}_4$ $\in$ $\Z^4$, такие что
\[F(\vec{z}_{1}) = \vec{z}_{3}, \, F(\vec{z}_{2}) = \vec{z}_{4}, \, F(\vec{z}_{3}) = \vec{z}_{1}, \, F(\vec{z}_{4}) = \vec{z}_{2}\] 
и выполняется хотя бы одно из следующих одиннадцати утверждений:

\textup{(1)} вектора $\vec{z}_{1}$, $\vec{z}_{2}$, $\vec{z}_{3}$, $\frac{1}{4}(\vec{z}_{1}+\vec{z}_{2}+\vec{z}_{3}+\vec{z}_{4})$ образуют базис решетки $\Z^4$;

\textup{(2)} вектора $\vec{z}_{1}$, $\vec{z}_{2}$,  $\vec{z}_{3}$, $\vec{z}_{4}$ образуют базис решетки $\Z^4$;

\textup{(3)} вектора $\vec{z}_{1}$, $\frac{1}{2}(\vec{z}_{1} + \vec{z}_{2})$, $\frac{1}{2}(\vec{z}_{1}+\vec{z}_{3})$, $\frac{1}{2}(\vec{z}_{1} + \vec{z}_{4})$ образуют базис решетки $\Z^4$;

\textup{(4)} вектора $\vec{z}_{1}$, $\vec{z}_{2}$, $\frac{1}{2}(\vec{z}_{1}+\vec{z}_{3})$, $\frac{1}{4}(\vec{z}_{1}+\vec{z}_{2}+\vec{z}_{3}+\vec{z}_{4})$ образуют базис решетки $\Z^4$;

\textup{(5)} вектора $\vec{z}_{1}$, $\vec{z}_{2}$, $\frac{1}{2}(\vec{z}_{1}+\vec{z}_{3})$, $\frac{1}{2}(\vec{z}_{2}+\vec{z}_{4})$ образуют базис решетки $\Z^4$;

\textup{(6)} вектора $\vec{z}_{1}$, $\vec{z}_{2}$, $\vec{z}_{3}$, $\frac{1}{2}(\vec{z}_{1} + \vec{z}_{3} + \vec{z}_{4} - \vec{z}_{2})$ образуют базис решетки $\Z^4$;

\textup{(7)} вектора $\vec{z}_{1}$, $\vec{z}_{2}$, $\vec{z}_{3}$, $\frac{1}{2}(\vec{z}_{1}+\vec{z}_{2}) + \frac{1}{4}(\vec{z}_{1}+\vec{z}_{4} - \vec{z}_{3} - \vec{z}_{2})$ образуют базис решетки $\Z^4$;

\textup{(8)} вектора $\vec{z}_{1}$, $\vec{z}_{2}$, $\vec{z}_{3}$, $\frac{1}{2}(\vec{z}_{2} + \vec{z}_{4})$ образуют базис решетки $\Z^4$;

\textup{(9)} вектора $\vec{z}_{1}$, $\vec{z}_{2}$, $\frac{1}{2}(\vec{z}_{1} + \vec{z}_{3})$, $\frac{1}{4}\vec{z}_{1} +  \frac{1}{2}\vec{z}_{2} - \frac{1}{4}\vec{z}_{3} + \frac{1}{2}\vec{z}_{4}$ образуют базис решетки $\Z^4$;

\textup{(10)} вектора $\vec{z}_{1}$, $\vec{z}_{2}$, $\frac{1}{2}(\vec{z}_{1}+\vec{z}_{3})$, $\frac{1}{2}\vec{z}_{1} +  \frac{1}{4}\vec{z}_{2} +  \frac{1}{2}\vec{z}_{4}$ образуют базис решетки $\Z^4$;

\textup{(11)} вектора $\vec{z}_{1}$, $\frac{1}{2}(\vec{z}_{1} + \vec{z}_{2})$, $\vec{z}_{3}$, $\frac{1}{2}(\vec{z}_{1} + \vec{z}_{3} + \vec{z}_{4} - \vec{z}_{2})$ образуют базис решетки $\Z^4$.
\end{lemma}
\begin{proof}
Будем называть плоскость \emph{рациональной}, если множество содержащихся в нем целых точек является (аффинной) решеткой ранга, равного размерности этой плоскости.

Рассмотрим для собственной симметрии $F$ подпространства $l^{1}_{+}$, $l^{2}_{+}$, $l^{1}_{-}$ и $l^{2}_{-}$ из леммы \ref{rational_subspace_2_2} и положим $S = l^{2}_{+} + l^{1}_{-} + l^{2}_{-}$. Обозначим через $S_1$ ближайшую к $S$ рациональную гиперплоскость, параллельную $S$ и не совпадающую с $S$ (любую из двух). Тогда $G(S_1) = S_1$. Также обозначим через $\vec{p}$ точку пересечения гиперплоскости $S_1$ и $l^{1}_{+}$, а через $l$ и $\pi$ прямую и плоскость, проходящие через точку $\vec{p}$ и параллельные $ l^{2}_{+}$ и $L_{-} = l^{1}_{-} + l^{2}_{-}$ соответственно. При этом $F(l) = l$, $F(\pi) = \pi$ и $F(\vec{p}) = \vec{p}$.

Плоскость $\pi$ разделяет гиперплоскость $S_1$ на два множества $S^{+}_1$ и $S^{-}_1$. Пусть $Q$ и $R$ --- рациональные плоскости ближайшие к $\pi$, параллельные $\pi$ и не совпадающие с $\pi$, принадлежащие множествам $S^{+}_1$ и $S^{-}_1$ соответственно. Отметим, что вообще говоря, расстояния от $\pi$ до $Q$ и от $\pi$ до $R$ не обязательно равны. Положим $\vec{p}^{Q} = Q \cap l$ и $\vec{p}^{R} = R \cap l$. 

Пусть  $(\vec{z}_{1} , \vec{z}_{2})$ --- такая пара точек решетки $\Z^4$, что $\vec{z}_{1} \in Q, \vec{z}_{2} \in R$, вектора $\vec{z}_{1} - \vec{p}^{Q}$ и $\vec{z}_{2} - \vec{p}^{R}$ неколлинеарны. Тогда можно построить точки 
\[\vec{z}_{3} = F(\vec{z}_{1}) \in \Z^4 \cap Q, \quad \vec{z}_{4} = F(\vec{z}_{2}) \in \Z^4  \cap R,\]
\[\vec{z}^{R}_{1} = \vec{z}_{1} + (\vec{p}^{R} - \vec{p}^{Q}), \vec{z}^{R}_{3} = \vec{z}_{3} + (\vec{p}^{R} - \vec{p}^{Q}), \vec{z}^{Q}_{2} = \vec{z}_{2} + (\vec{p}^{Q} - \vec{p}^{R}), \vec{z}^{Q}_{4} = \vec{z}_{4} + (\vec{p}^{Q} - \vec{p}^{R}),\]
\[\vec{z}^{\pi}_{1} = \vec{z}_{1} + (\vec{p} - \vec{p}^{Q}), \vec{z}^{\pi}_{3} = \vec{z}_{3} + (\vec{p} - \vec{p}^{Q}), \vec{z}^{\pi}_{2} = \vec{z}_{2} + (\vec{p} - \vec{p}^{R}), \vec{z}^{\pi}_{4} = \vec{z}_{4} + (\vec{p} - \vec{p}^{R}).\]
Рассмотренные точки определяют тройку параллелограммов $(\Delta^{\pi}, \Delta^{Q}, \Delta^{R})$, где
\[\Delta^{\pi} = \textup{conv}(\vec{z}^{\pi}_{1}, \vec{z}^{\pi}_{2}, \vec{z}^{\pi}_{3}, \vec{z}^{\pi}_{4}) \subset \pi,\] 
\[\Delta^{Q} = \textup{conv}(\vec{z}_{1}, \vec{z}^{Q}_{2}, \vec{z}_{3}, \vec{z}^{Q}_{4}) \subset Q, \quad \Delta^{R} = \textup{conv}(\vec{z}^{R}_{1}, \vec{z}_{2}, \vec{z}^{R}_{3}, \vec{z}_{4}) \subset R.\]

При помощи метода спуска, можно построить такую пару $(\vec{z}_{1}, \vec{z}_{2})$, что (см. рисунок \ref{integer_points})
\[ \Delta^{Q} \cap \Z^4 \subseteq \{\vec{z}_{1}, \vec{z}^{Q}_{2}, \vec{z}_{3}, \vec{z}^{Q}_{4}, \vec{p}^{Q}, \frac{\vec{z}_{2}^{Q} + \vec{p}^{Q}}{2},  \frac{\vec{z}_{4}^{Q} + \vec{p}^{Q}}{2}\},\]
\[ \Delta^{R} \cap \Z^4 \subseteq \{\vec{z}^{R}_{1}, \vec{z}_{2}, \vec{z}^{R}_{3}, \vec{z}_{4}, \vec{p}^{R}, \frac{\vec{z}_{1}^{R} + \vec{p}^{R}}{2},  \frac{\vec{z}_{3}^{R} + \vec{p}^{R}}{2}\}, \]
и $\Delta^{\pi} \cap \Z^4$ является подмножеством множества
\scriptsize{\[  \{\vec{z}^{\pi}_{1}, \vec{z}^{\pi}_{2}, \vec{z}^{\pi}_{3}, \vec{z}^{\pi}_{4}, \vec{p}, \frac{\vec{z}_{1}^{\pi} + \vec{z}_{2}^{\pi}}{2},  \frac{\vec{z}_{2}^{\pi} + \vec{z}_{3}^{\pi}}{2}, \frac{\vec{z}_{3}^{\pi} + \vec{z}_{4}^{\pi}}{2}, \frac{\vec{z}_{4}^{\pi} + \vec{z}_{1}^{\pi}}{2}, \frac{\vec{z}^{\pi}_{1} +  \vec{p}}{2}, \frac{\vec{z}^{\pi}_{2} +  \vec{p}}{2}, \frac{\vec{z}^{\pi}_{3} +  \vec{p}}{2}, \frac{\vec{z}^{\pi}_{4} +  \vec{p}}{2}\}.\]}
\normalsize
Аккуратно перебирая возможные расположения точек решетки $\Z^4$ в тройке параллелограммов $(\Delta^{\pi}, \Delta^{Q}, \Delta^{R})$, мы попадаем в одну из одиннадцати ситуаций, соответствующей одному из утверждений (1) - (11).

 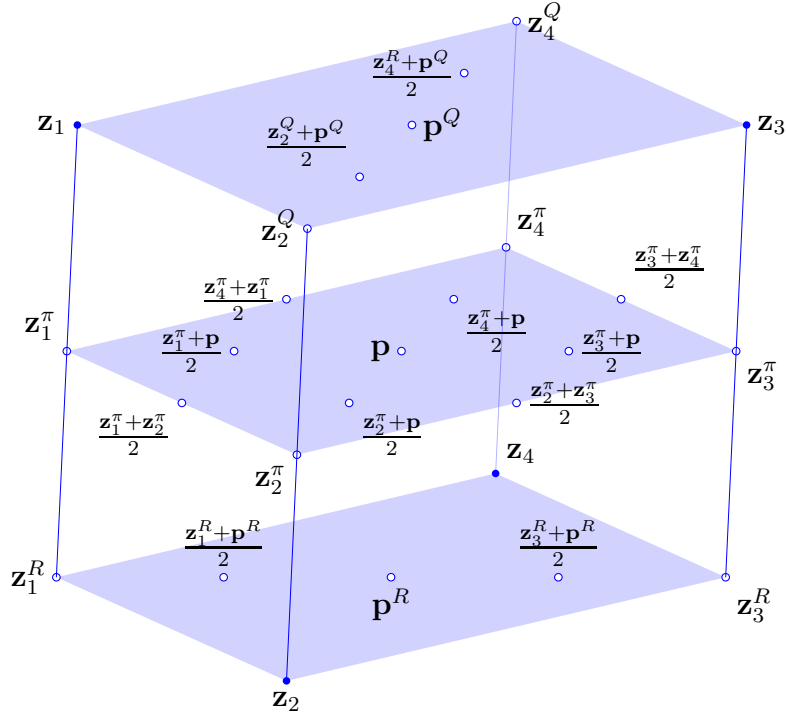
\begin{figure}[h]
  \centering
  \begin{tikzpicture}[x=10mm, y=7mm, z=-5mm, scale=1.1]
    \begin{scope}[rotate around z=0]

    \coordinate (z_1) at (-6,3,0);
    \coordinate (z_3) at (2,3,0);
    \coordinate (z_2) at (-1,-3,5);
    \coordinate (z_4) at (-1,-3,0);
    
    \coordinate (p_pl_1) at ($ (z_1)!1/2!(z_3) $);
    \coordinate (p_mi_1) at ($ (z_2)!1/2!(z_4) $);
    \coordinate (z_1_z_2_half) at ($ (z_1)!1/2!(z_2) $);
    \coordinate (z_2_z_3_half) at ($ (z_2)!1/2!(z_3) $);
    \coordinate (z_3_z_4_half) at ($ (z_3)!1/2!(z_4) $);
    \coordinate (z_4_z_1_half) at ($ (z_4)!1/2!(z_1) $); 
    
    \coordinate (p) at ($ (p_pl_1)!1/2!(p_mi_1) $);
    \coordinate (z_1_pl_1) at ($(p_mi_1) + (z_1) - (p_pl_1) $);
    \coordinate (z_3_pl_1) at ($(p_mi_1) + (z_3) - (p_pl_1) $);
    \coordinate (z_2_mi_1) at ($(p_pl_1) + (z_2) - (p_mi_1) $);
    \coordinate (z_4_mi_1) at ($(p_pl_1) + (z_4) - (p_mi_1) $);
    
    \coordinate (z_1_p) at ($ (z_1)!1/2!(z_1_pl_1) $);
    \coordinate (z_3_p) at ($(z_3)!1/2!(z_3_pl_1)$);
    \coordinate (z_2_p) at ($(z_2)!1/2!(z_2_mi_1)$);
    \coordinate (z_4_p) at ($(z_4)!1/2!(z_4_mi_1)$);
    
    \coordinate (p_1_p) at ($(z_1_pl_1)!1/2!(p_mi_1)$);
    \coordinate (p_3_p) at ($(z_3_pl_1)!1/2!(p_mi_1)$);
    
    \coordinate (p_2_q) at ($(z_2_mi_1)!1/2!(p_pl_1)$);
    \coordinate (p_4_q) at ($(z_4_mi_1)!1/2!(p_pl_1)$);
    
    \coordinate (z_1_p_1_2) at ($(p)!1/2!(z_1_p)$);
    \coordinate (z_3_p_1_2) at ($(p)!1/2!(z_3_p)$);
    \coordinate (z_2_p_1_2) at ($(p)!1/2!(z_2_p)$);
    \coordinate (z_4_p_1_2) at ($(p)!1/2!(z_4_p)$);

     \fill[blue!25,opacity=0.7] (z_1) -- (z_2_mi_1) -- (z_3) -- (z_4_mi_1) -- cycle;
     
      \fill[blue!25,opacity=0.7] (z_1_pl_1) -- (z_2) -- (z_3_pl_1) -- (z_4) -- cycle;
      
     \fill[blue!25,opacity=0.7] (z_1_p) -- (z_2_p) -- (z_3_p) -- (z_4_p) -- cycle;
    
     \node[fill=blue,circle,inner sep=1pt] at (z_1) {};
    \draw (z_1) node[left] {$\vec{z}_{1}$};
    
    \node[fill=blue,circle,inner sep=1pt] at (z_3) {};
    \draw (z_3) node[right] {$\vec{z}_{3}$};
    
    \node[fill=blue,circle,inner sep=1pt] at (z_2) {};
    \draw (z_2) node[below] {$\vec{z}_{2}$};
    
     \node[fill=blue,circle,inner sep=1pt] at (z_4) {};
    \draw (z_4) node[above right] {$\vec{z}_{4}$};

    \node[fill=white,circle,inner sep=1pt,draw=blue] at (z_1_pl_1) {};
    \draw (z_1_pl_1) node[left] {$\vec{z}_{1}^{R}$};
    
    \node[fill=white,circle,inner sep=1pt,draw=blue] at (z_3_pl_1) {};
    \draw (z_3_pl_1) node[below right] {$\vec{z}_{3}^{R}$};
    
    \node[fill=white,circle,inner sep=1pt,draw=blue] at (z_2_mi_1) {};
    \draw (z_2_mi_1) node[left] {$\vec{z}_{2}^{Q}$};
    
    \node[fill=white,circle,inner sep=1pt,draw=blue] at (z_4_mi_1) {};
    \draw (z_4_mi_1) node[right] {$\vec{z}_{4}^{Q}$};
    
     \node[fill=white,circle,inner sep=1pt,draw=blue] at (z_1_p) {};
    \draw (z_1_p) node[above left] {$\vec{z}_{1}^{\pi}$};
    
    \node[fill=white,circle,inner sep=1pt,draw=blue] at (z_3_p) {};
    \draw (z_3_p) node[below right] {$\vec{z}_{3}^{\pi}$};
    
     \node[fill=white,circle,inner sep=1pt,draw=blue] at (z_2_p) {};
    \draw (z_2_p) node[below left] {$\vec{z}_{2}^{\pi}$};
    
     \node[fill=white,circle,inner sep=1pt,draw=blue] at (z_4_p) {};
    \draw (z_4_p) node[above right] {$\vec{z}_{4}^{\pi}$};
    
     \node[fill=white,circle,inner sep=1pt,draw=blue] at (p) {};
    \draw (p) node[left] {$\vec p$};
    
    
    \node[fill=white,circle,inner sep=1pt,draw=blue] at (p_pl_1) {};
    \draw (p_pl_1) node[right] {$\vec{p}^{Q}$};
    
    \node[fill=white,circle,inner sep=1pt,draw=blue] at (p_mi_1) {};
    \draw (p_mi_1) node[below] {$\vec{p}^{R}$};
    
     \node[fill=white,circle,inner sep=1pt,draw=blue] at (z_1_z_2_half) {};
    \draw (z_1_z_2_half) node[below left] {$\frac{\vec{z}_{1}^{\pi} + \vec{z}_{2}^{\pi}}{2}$};
    
    \node[fill=white,circle,inner sep=1pt,draw=blue] at (z_2_z_3_half) {};
    \draw (z_2_z_3_half) node[right] {$\frac{\vec{z}_{2}^{\pi} + \vec{z}_{3}^{\pi}}{2}$};
    
         \node[fill=white,circle,inner sep=1pt,draw=blue] at (z_3_z_4_half) {};
    \draw (z_3_z_4_half) node[above right] {$\frac{\vec{z}_{3}^{\pi} + \vec{z}_{4}^{\pi}}{2}$};
    
     \node[fill=white,circle,inner sep=1pt,draw=blue] at (z_4_z_1_half) {};
    \draw (z_4_z_1_half) node[left] {$\frac{\vec{z}_{4}^{\pi} + \vec{z}_{1}^{\pi}}{2}$};
    
    \node[fill=white,circle,inner sep=1pt,draw=blue] at (p_1_p) {};
    \draw (p_1_p) node[above] {$\frac{\vec{z}_{1}^{R} + \vec{p}^{R}}{2}$};
    
    \node[fill=white,circle,inner sep=1pt,draw=blue] at (p_3_p) {};
    \draw (p_3_p) node[above] {$\frac{\vec{z}_{3}^{R} + \vec{p}^{R}}{2}$};
    
      \node[fill=white,circle,inner sep=1pt,draw=blue] at (p_2_q) {};
    \draw (p_2_q) node[above left] {$\frac{\vec{z}_{2}^{Q} + \vec{p}^{Q}}{2}$};
    
    \node[fill=white,circle,inner sep=1pt,draw=blue] at (p_4_q) {};
    \draw (p_4_q) node[left] {$\frac{\vec{z}_{4}^{R} + \vec{p}^{Q}}{2}$};
    
     \node[fill=white,circle,inner sep=1pt,draw=blue] at (z_1_p_1_2) {};
    \draw (z_1_p_1_2) node[left] {$\frac{\vec{z}^{\pi}_{1} +  \vec{p}}{2}$};
    
    \node[fill=white,circle,inner sep=1pt,draw=blue] at (z_3_p_1_2) {};
    \draw (z_3_p_1_2) node[right] {$\frac{\vec{z}^{\pi}_{3} +  \vec{p}}{2}$};
    
    \node[fill=white,circle,inner sep=1pt,draw=blue] at (z_2_p_1_2) {};
    \draw (z_2_p_1_2) node[below right] {$\frac{\vec{z}^{\pi}_{2} +  \vec{p}}{2}$};
    
    \node[fill=white,circle,inner sep=1pt,draw=blue] at (z_4_p_1_2) {};
    \draw (z_4_p_1_2) node[below right] {$\frac{\vec{z}^{\pi}_{4} +  \vec{p}}{2}$};
    
   \draw[-] [blue,very thin] (z_1_pl_1) -- (z_1);
   \draw[-] [blue,very thin] (z_3_pl_1) -- (z_3);
   \draw[-] [blue,very thin, opacity=2] (z_2_mi_1) -- (z_2);
   \draw[-] [blue,very thin, opacity=0.3] (z_4_mi_1) -- (z_4);

    \end{scope}
  \end{tikzpicture}
  \caption{Возможное расположение точек решетки $\Z^4$ в параллелограммах из построенной тройки $(\Delta^{\pi}, \Delta^{Q}, \Delta^{R})$}
  \label{integer_points}
\end{figure}

\end{proof}

 \section{Матрицы собственных симметрий}\label{matrix_and_algebraicity}

Напомним, что если задана дробь $\cf(l_1,l_2,l_3,l_4)=\cf(A)\in\gA_3$, будем считать, что подпространство $l_1$ порождается вектором $\vec l_1=(1,\alpha,\beta,\gamma)$. Тогда из предложения \ref{prop:more_than_pelle_n_dim} следует, что числа $1,\alpha,\beta, \gamma$ образуют базис поля $K=\Q(\alpha,\beta, \gamma)$ над $\Q$ и каждое $l_i$ порождается вектором $\vec l_i=(1,\sigma_i(\alpha),\sigma_i(\beta), \sigma_i(\gamma))$, где $\sigma_1(=\id),\sigma_2,\sigma_3, \sigma_4$ --- все вложения $K$ в $\R$. То есть, если через $\big(\vec l_1,\vec l_2,\vec l_3,\vec l_4\big)$ обозначить матрицу со столбцами $\vec l_1,\vec l_2,\vec l_3,\vec l_4$, получим
 \[
  \big(\vec l_1,\vec l_2,\vec l_3,\vec l_4\big)=
  \begin{pmatrix}
    1 & 1 & 1 & 1 \\
    \alpha & \sigma_2(\alpha) & \sigma_3(\alpha) & \sigma_4(\alpha) \\
    \beta & \sigma_2(\beta) & \sigma_3(\beta) & \sigma_4(\beta) \\
    \gamma & \sigma_2(\gamma) & \sigma_3(\gamma) & \sigma_4(\gamma)
  \end{pmatrix}.
\]

Мы будем обозначать через $\widetilde{\gA_{3}}$ множество всех трехмерных алгебраических цепных дробей, для которых 
\[\sigma_3(K) = K, \quad \sigma_4(K) = \sigma_2(K), \quad \sigma^{2}_{3} = \textup{id}, \quad \sigma_{4} = \sigma_{2}\sigma_{3}.\]
 
Для каждого $i=1, 2, \ldots, 10$ определим $\mathbf{CF}_{i}$ как класс дробей из $\widetilde{\gA_{3}}$, удовлетворяющих паре соотношений $\gR_{i}$, где
  
  $\gR_{1}$: $\beta + \sigma_{3}(\beta)= -\big(\alpha + \sigma_{3}(\alpha)\big),\ \gamma =   \sigma_{3}(\alpha)$;
  
  $\gR_{2}$: $\beta + \sigma_{3}(\beta)= 1 - \big(\alpha + \sigma_{3}(\alpha)\big) ,\ \gamma = \sigma_{3}(\alpha)$;
  
  $\gR_{3}$: $\beta + \sigma_{3}(\beta)= 2 - \big(\alpha + \sigma_{3}(\alpha)\big) ,\ \gamma = \sigma_{3}(\alpha)$;
  
  $\gR_{4}$: $\beta + \sigma_{3}(\beta)= -\big(\alpha + \sigma_{3}(\alpha)\big),\ \gamma =  \frac{\alpha + \sigma_{3}(\alpha)}{2}$;
  
  $\gR_{5}$: $\beta + \sigma_{3}(\beta)= 2 -\big(\alpha + \sigma_{3}(\alpha)\big),\ \gamma =   \frac{\alpha + \sigma_{3}(\alpha)}{2}$;
  
  $\gR_{6}$: $\beta + \sigma_{3}(\beta)= -\big(\alpha + \sigma_{3}(\alpha)\big),\ \gamma =  \frac{\alpha + \sigma_{3}(\alpha) + 1}{2}$;
  
  $\gR_{7}$: $\beta + \sigma_{3}(\beta)= 2 -\big(\alpha + \sigma_{3}(\alpha)\big),\ \gamma =  \frac{\alpha + \sigma_{3}(\alpha) + 1}{2}$;
  
  $\gR_{8}$: $\beta + \sigma_{3}(\beta)= 1 - \frac{\alpha + \sigma_{3}(\alpha)}{2},\ \gamma =  \frac{\alpha + \sigma_{3}(\alpha)}{2}$;
  
  $\gR_{9}$: $\beta + \sigma_{3}(\beta)= 2 - \frac{\alpha + \sigma_{3}(\alpha)}{2},\ \gamma =  \frac{\alpha + \sigma_{3}(\alpha)}{2}$;
  
  $\gR_{10}$: $\beta + \sigma_{3}(\beta)= 2 - \frac{\alpha + \sigma_{3}(\alpha)}{2},\ \gamma =  \frac{\sigma_{3}(\alpha) - \alpha}{4}$.

Покажем, что все дроби из классов $\mathbf{CF}_{i}$, палиндромичны для каждого $i=1, 2, \ldots, 10$. Положим $G_{1}, G_{2}, \ldots, G_{10}$ равными соответственно матрицам 
\[
  \left(\begin{smallmatrix}
    \phantom{-}1 & \phantom{-}0 & \phantom{-}0 & \phantom{-}0 \\
    \phantom{-}0 & \phantom{-}0 & \phantom{-}0 & \phantom{-}1 \\
    \phantom{-}0 &                  -1 &                   -1 &                  -1 \\
    \phantom{-}0 & \phantom{-}1 & \phantom{-}0 & \phantom{-}0
  \end{smallmatrix}\right),
  \left(\begin{smallmatrix}
    \phantom{-}1 & \phantom{-}0 & \phantom{-}0 & \phantom{-}0 \\
    \phantom{-}0 & \phantom{-}0 & \phantom{-}0 & \phantom{-}1 \\
    \phantom{-}1 &                  -1 &                   -1 &                  -1 \\
    \phantom{-}0 & \phantom{-}1 & \phantom{-}0 & \phantom{-}0
  \end{smallmatrix}\right),
  \left(\begin{smallmatrix}
    1 & \phantom{-}0 & \phantom{-}0 & \phantom{-}0 \\
    0 & \phantom{-}0 & \phantom{-}0 & \phantom{-}1 \\
    2 &                  -1 &                   -1 &                  -1 \\
    0 & \phantom{-}1 & \phantom{-}0 & \phantom{-}0
  \end{smallmatrix}\right),
  \left(\begin{smallmatrix}
    1 & \phantom{-}0 & \phantom{-}0 & \phantom{-}0 \\
    0 &                  -1 & \phantom{-}0 & \phantom{-}2 \\
    0 & \phantom{-}0 &                   -1 &                  -2 \\
    0 & \phantom{-}0 & \phantom{-}0 & \phantom{-}1
  \end{smallmatrix}\right),
  \left(\begin{smallmatrix}
    1 & \phantom{-}0 & \phantom{-}0 & \phantom{-}0 \\
    0 &                  -1 & \phantom{-}0 & \phantom{-}2 \\
    2 & \phantom{-}0 &                   -1 &                  -2 \\
    0 & \phantom{-}0 & \phantom{-}0 & \phantom{-}1
  \end{smallmatrix}\right),
 \]
 \[
  \left(\begin{smallmatrix}
    \phantom{-}1 & \phantom{-}0 & \phantom{-}0 & \phantom{-}0 \\
                     -1 &                   -1 & \phantom{-}0 & \phantom{-}2 \\
    \phantom{-}1 & \phantom{-}0 &                   -1 &                  -2 \\
    \phantom{-}0 & \phantom{-}0 & \phantom{-}0 & \phantom{-}1
  \end{smallmatrix}\right),
  \left(\begin{smallmatrix}
    \phantom{-}1 & \phantom{-}0 & \phantom{-}0 & \phantom{-}0 \\
                     -1 &                  -1 & \phantom{-}0 & \phantom{-}2 \\
    \phantom{-}3 & \phantom{-}0 &                   -1 &                  -2 \\
    \phantom{-}0 & \phantom{-}0 & \phantom{-}0 & \phantom{-}1
  \end{smallmatrix}\right),
  \left(\begin{smallmatrix}
    1 & \phantom{-}0 & \phantom{-}0 & \phantom{-}0 \\
    0 &                  -1 & \phantom{-}0 & \phantom{-}2 \\
    1 & \phantom{-}0 &                   -1 &                  -1 \\
    0 & \phantom{-}0 & \phantom{-}0 & \phantom{-}1
  \end{smallmatrix}\right),
  \left(\begin{smallmatrix}
    1 & \phantom{-}0 & \phantom{-}0 & \phantom{-}0 \\
    0 &                  -1 & \phantom{-}0 & \phantom{-}2 \\
    2 & \phantom{-}0 &                   -1 &                  -1 \\
    0 & \phantom{-}0 & \phantom{-}0 & \phantom{-}1
  \end{smallmatrix}\right),
  \left(\begin{smallmatrix}
    1 & \phantom{-}0 & \phantom{-}0 & \phantom{-}0 \\
    0 & \phantom{-}1 & \phantom{-}0 & \phantom{-}4 \\
    2 &                  -1 &                   -1 &                  -2 \\
    0 & \phantom{-}0 & \phantom{-}0 &                  -1
  \end{smallmatrix}\right).
\]

\begin{lemma}\label{oper_eq_4d_2}
  Пусть $\cf(l_1,l_2,l_3,l_4)\in\gA_3$ и $i\in\{1,2, \ldots,10\}$. Тогда цепная дробь $\cf(l_1,l_2,l_3,l_4)$ принадлежит классу $\mathbf{CF}_{i}$ в том и только в том случае, если $G_{i}$ --- её собственная симметрия и $\textup{ord}({\sigma_{G_{i}}}) = 2$.
\end{lemma}

\begin{proof}

Покажем, что $\cf(l_1,l_2,l_3,l_4)$ принадлежит классу $\mathbf{CF}_{1}$ в том и только в том случае, если $G_{1}$ --- её собственная симметрия и $\textup{ord}({\sigma_{G_{1}}}) = 2$.

  В силу леммы \ref{prod_lemma}, оператор $G\in\Gl_4(\Z)$ является собственной симметрией дроби $\cf(l_1,l_2,l_3,l_4)$ и $\textup{ord}({\sigma_{G}}) = 2$ тогда и только тогда, когда с точностью до перестановки индексов существуют такие действительные  числа $\mu_1,\mu_2,\mu_3,\mu_4$, что $G\big(\vec l_1,\vec l_2,\vec l_3,\vec l_4\big)=\big(\mu_3\vec l_3,\mu_4\vec l_4,\mu_1\vec l_1,\mu_2\vec l_2\big)$ и $\mu_1\mu_3 = \mu_2\mu_4 = 1$.
  
Пусть $\cf(l_1, l_2, l_3, l_4) \in \mathbf{CF}_{1}$. Заметим, что $\sigma_{2} = \sigma_{2}\sigma^{2}_{3} = \sigma_{4}\sigma_{3}$, $\sigma_2(\gamma) = \sigma_{2}\sigma_{3}(\alpha)$, $\sigma_3(\gamma) = \sigma^{2}_{3}(\alpha) = \alpha$, $\sigma_4(\gamma) = \sigma_{4}\sigma_{3}(\alpha) = \sigma_{2}(\alpha)$, $\sigma_{2}(\beta) + \sigma_{2}\sigma_{3}(\beta) = \sigma_{2}\big(\beta + \sigma_{3}(\beta)\big) =  \sigma_{2}\Big( - \big(\alpha + \sigma_{3}(\alpha)\big)\Big) = - \big(\sigma_{2}(\alpha) + \sigma_{2}\sigma_{3}(\alpha)\big)$ и $\sigma_{4}(\beta) = \sigma_{2}\sigma_{3}(\beta)$. Тогда

\[    G_{1} \big(\vec{l}_1, \vec{l}_2, \vec{l}_3, \vec{l}_4\big) = \]
 \[  \left( \tiny \begin{smallmatrix}
      1\phantom{-}\phantom{-} & 1\phantom{-}\phantom{-} & 1\phantom{-}\phantom{-} & 1 \\
      \sigma_{3}(\alpha)\phantom{-}\phantom{-} &  \sigma_{2}\sigma_{3}(\alpha)\phantom{-}\phantom{-} &  \alpha\phantom{-}\phantom{-} & \sigma_{2}(\alpha) \\
      - \beta - \big(\alpha + \sigma_{3}(\alpha)\big)\phantom{-}\phantom{-} &  - \sigma_{2}(\beta) - \big(\sigma_{2}(\alpha) + \sigma_{2}\sigma_{3}(\alpha)\big)\phantom{-}\phantom{-} &  - \sigma_{3}(\beta) - \big(\sigma_{3}(\alpha) + \alpha\big)\phantom{-}\phantom{-} & - \sigma_{2}\sigma_{3}(\beta) - \big(\sigma_{2}\sigma_{3}(\alpha) + \sigma_{2}(\alpha)\big) \\
      \alpha\phantom{-}\phantom{-} & \sigma_{2}(\alpha)\phantom{-}\phantom{-} & \sigma_{3}(\alpha)\phantom{-}\phantom{-} & \sigma_{2}\sigma_{3}(\alpha) 
    \end{smallmatrix}\right) = \]
\[ \big(\vec{l}_3, \vec{l}_4, \vec{l}_1, \vec{l}_2\big). \]
Следовательно, $G_{1}$ ---  собственная симметрия $\cf(l_1,l_2,l_3,l_4)$ и $\textup{ord}({\sigma_{G_{1}}}) = 2$. Обратно, предположим, $G_{1}$ ---  собственная симметрия $\cf(l_1,l_2,l_3,l_4)$ и $\textup{ord}({\sigma_{G_{1}}}) = 2$. Тогда существует $\mu_3$, такое что с точностью до перестановки индексов
  \[
    G_{1}\vec{l}_1=
    \begin{pmatrix}
      1  \\
      \gamma \\
      - \beta - (\alpha + \gamma) \\
      \alpha
    \end{pmatrix}
    =\mu_3
    \begin{pmatrix}
      1  \\
      \sigma_{3}(\alpha) \\
      \sigma_{3}(\beta) \\
      \sigma_{3}(\gamma)
    \end{pmatrix}, 
  \]
  откуда $\mu_3 = 1$,  $\gamma =\sigma_{3}(\alpha)$, $\sigma^{2}_3(\alpha)=\sigma_{3}(\gamma) = \alpha$, $\sigma^{2}_{3}(\gamma) = \sigma_{3}(\alpha) = \gamma$, $\beta + \sigma_{3}(\beta) = - \big(\alpha + \sigma_{3}(\alpha)\big)$, $\sigma^{2}_3(\beta) = - \sigma_3(\beta)  - \big(\sigma_3(\alpha) + \sigma^{2}_{3}(\alpha)\big) = - \sigma_3(\beta) - \big(\alpha + \sigma_3(\alpha)\big) = \beta$. Существует $\mu_4$, такое что
  \[ 
    G_{1}\vec{l}_2=
    \begin{pmatrix}
      1  \\
      \sigma_{2}(\gamma) \\
      -  \sigma_{2}(\beta) - \big(\sigma_{2}(\alpha) +  \sigma_{2}(\gamma)\big) \\
       \sigma_{2}(\alpha)
    \end{pmatrix}
    =\mu_4
    \begin{pmatrix}
      1  \\
      \sigma_{4}(\alpha) \\
      \sigma_{4}(\beta) \\
      \sigma_{4}(\gamma)
    \end{pmatrix},
  \]
   откуда $\mu_4= 1$, $\sigma_{4}(\alpha) = \sigma_{2}(\gamma) = \sigma_{2}\sigma_{3}(\alpha)$, $\sigma_{4}(\gamma) = \sigma_{2}(\alpha) = \sigma_{2}\sigma_{3}(\gamma)$, $\sigma_{4}(\beta) = -  \sigma_{2}(\beta) - \big(\sigma_{2}(\alpha) +  \sigma_{2}(\gamma)\big) = \sigma_{2}\big(-  \beta - (\alpha + \gamma)\big) = \sigma_{2}\sigma_{3}(\beta)$. Следовательно, $\cf(l_1, l_2, l_3, l_4) \in \mathbf{CF}_{1}$, так как числа $1,\alpha,\beta, \gamma$ образуют базис поля $K=\Q(\alpha,\beta, \gamma)$.
  
  Для $i=2, \ldots, 10$ рассуждения аналогичны.

\end{proof}

\section{Доказательство теоремы \ref{theorem_proper}}\label{proof_theorem_proper}
 
Обозначим для каждого $i=1,\ldots,10$ через $\overline{\mathbf{CF}}_i$ образ $\mathbf{CF}_i$ при действии группы $\Gl_4(\Z)$:
\[
  \overline{\mathbf{CF}}_i=
  \Big\{ \cf(l_1,l_2,l_3,l_4)\in\gA_3 \,\Big|\, \exists X\in\Gl_4(\Z):X\big(\cf(l_1,l_2,l_3,l_4)\big)\in\mathbf{CF}_i \Big\}.
\]

\begin{lemma}\label{l:CF_instead_of_statements}
  Для дроби $\cf(l_1,l_2,l_3,l_4)\in\gA_3$ выполняется условие $(i)$ теоремы \ref{theorem_proper} тогда и только тогда, когда $\cf(l_1,l_2,l_3,l_4)$ принадлежит классу $\overline{\mathbf{CF}}_i$, где $i\in\{1, 2,\ldots, 10\}$.
\end{lemma}

\begin{proof}
  Для любого $X\in\Gl_4(\Z)$ гиперболичность оператора $A\in\Gl_4(\Z)$ равносильна гиперболичности оператора $XAX^{-1}$. При этом собственные подпространства гиперболического оператора однозначно восстанавливаются по любому его собственному вектору. Остаётся воспользоваться определением эквивалентности из параграфа \ref{intro}.
\end{proof}

Теорему \ref{theorem_proper} при помощи леммы \ref{l:CF_instead_of_statements}  можно переформулировать следующим образом: \emph{дробь $\cf(l_1,l_2,l_3,l_4)\in\gA_3$ имеет собственную симметрию $G$ тогда и только тогда, когда она принадлежит одному из классов $\overline{\mathbf{CF}}_i$, где $i\in\{1,2,\ldots,10\}$}.

  \begin{proof}[теоремы \ref{theorem_proper}]
Если $\cf(l_1,l_2,l_3,l_4)$ принадлежит какому-то $\overline{\mathbf{CF}}_i$, то по лемме \ref{oper_eq_4d_2} она имеет собственную симметрию $G$, ибо действие оператора из $\Gl_4(\Z)$ сохраняет свойство существования у алгебраической цепной дроби собственной симметрии.

Обратно, пусть дробь $\cf(l_1,l_2,l_3,l_4)\in\gA_3$ имеет собственную симметрию $G$. Положим $F=G'$ и рассмотрим точки $\vec z_1$, $\vec z_2$, $\vec z_3$, $\vec z_4$ из леммы \ref{main_lem_2}. Обозначим также через $\vec{e}_1$, $\vec{e}_2$, $\vec{e}_3$, $\vec{e}_4$ стандартный базис $\R^4$. Для точек $\vec z_1$, $\vec z_2$, $\vec z_3$, $\vec z_4$ выполняется хотя бы одно из утверждений \textup{(1)} - \textup{(11)} леммы \ref{main_lem_2}.

Пусть выполняется утверждение \textup{(1)} леммы \ref{main_lem_2}. Рассмотрим такой оператор $X_{1} \in \Gl_4(\Z)$, что
\[X_{1}\big(\vec{z}_1, \vec{z}_2, \vec{z}_3, \frac{1}{4}(\vec{z}_{1}+\vec{z}_{2}+\vec{z}_{3}+\vec{z}_{4})\big) = \big(\vec{e}_1 - \vec{e}_2, \vec{e}_1 + \vec{e}_4, \vec{e}_1 + \vec{e}_3 - \vec{e}_4, \vec{e}_1\big).\]
Тогда $X_{1}(\vec{z}_4) = X_{1}(4\cdot\frac{1}{4}(\vec{z}_{1}+\vec{z}_{2}+\vec{z}_{3}+\vec{z}_{4}) - \vec{z}_1 - \vec{z}_2 - \vec{z}_3) = \vec{e}_1 + \vec{e}_2 - \vec{e}_3$ и $X_{1}GX_{1}^{-1} = G_{1}$, так как по лемме \ref{main_lem_2}
\[X_{1}GX_{1}^{-1}\big(\vec{e}_1 - \vec{e}_2, \vec{e}_1 + \vec{e}_4, \vec{e}_1 + \vec{e}_3 - \vec{e}_4, \vec{e}_1 + \vec{e}_2 - \vec{e}_3\big) = \]
\[\big(\vec{e}_1 + \vec{e}_3 - \vec{e}_4, \vec{e}_1 + \vec{e}_2 - \vec{e}_3, \vec{e}_1 - \vec{e}_2, \vec{e}_1 + \vec{e}_4\big).\]
Стало быть, $X_{1}\big(\cf(l_1,l_2,l_3,l_4)\big) \in  \mathbf{CF}_{1}$, то есть $\cf(l_1,l_2,l_3,l_4)\in\overline{\mathbf{CF}}_1$.

Рассуждения аналогичны для случаев, когда выполняется утверждение ($i$) леммы \ref{main_lem_2}, где $i=2, \ldots, 10$.

Если выполняется утверждение \textup{(11)} леммы \ref{main_lem_2}, то, как и при выполнении утверждения \textup{(2)} леммы \ref{main_lem_2}, вновь $\cf(l_1,l_2,l_3,l_4)\in\overline{\mathbf{CF}}_2$. Действительно, пусть выполняется утверждение \textup{(11)} леммы \ref{main_lem_2}. Рассмотрим такой оператор $X_{11} \in \Gl_4(\Z)$, что
\[X_{11}\big(\vec{z}_1, \frac{1}{2}(\vec{z}_1 + \vec{z}_2), \vec{z}_3, \frac{1}{2}(\vec{z}_1 - \vec{z}_2 + \vec{z}_3 + \vec{z}_4)\big) = \big(\vec{e}_1, \vec{e}_1 + \vec{e}_4, \vec{e}_1 + \vec{e}_3, \vec{e}_1 + \vec{e}_2 - \vec{e}_4\big).\]
Тогда $X_{11}(\vec{z}_2) = X_{11}(2\cdot\frac{1}{2}(\vec{z}_1 + \vec{z}_2) - \vec{z}_1) = \vec{e}_1 + 2\vec{e}_4$, $X_{11}(\vec{z}_4) = X_{11}(2\cdot\frac{1}{2}(\vec{z}_1 - \vec{z}_2 + \vec{z}_3 + \vec{z}_4) - \vec{z}_1 + \vec{z}_2 -  \vec{z}_3) = \vec{e}_1 + 2\vec{e}_2 - \vec{e}_3$ и $X_{11}GX_{11}^{-1} = G_{2}$, так как по лемме \ref{main_lem_2}
\[X_{11}GX_{11}^{-1}\big(\vec{e}_1, \vec{e}_1 + 2\vec{e}_4, \vec{e}_1 + \vec{e}_3, \vec{e}_1 + 2\vec{e}_2 - \vec{e}_3\big) = \]
\[\big(\vec{e}_1 + \vec{e}_3, \vec{e}_1 + 2\vec{e}_2 - \vec{e}_3, \vec{e}_1, \vec{e}_1 + 2\vec{e}_4\big).\]
Стало быть, $X_{11}\big(\cf(l_1,l_2,l_3,l_4)\big) \in  \mathbf{CF}_{2}$, то есть $\cf(l_1,l_2,l_3,l_4)\in\overline{\mathbf{CF}}_2$.
\end{proof}

\section{Пример палиндромичной цепной дроби, не обладающей собственными циклическими симметриями}\label{proof_proper_but_not_cyclic}

 Рассмотрим вещественные числа $\theta_{1} = \sqrt{4 + \sqrt{2}}$ и $\theta_{2} = \sqrt{4 - \sqrt{2}}$. Заметим, что $\theta_{1}^{2} = 4 +  \sqrt{2}$, а значит $\theta_{1}$ является корнем уравнения $x^{4} - 8 x^{2} + 14 = 0$. В силу критерия Эйзенштейна для $p=2$ многочлен $f(x) = x^{4} - 8 x^{2} + 14$ неприводим над $\Q$. Таким образом, $f(x)$ --- минимальный многочлен для $\theta_{1}$ и 
 \[f(x) = (x - \theta_{1})(x - \theta_{2})(x + \theta_{1})(x + \theta_{2}).\]
 Рассмотрим такие вложения  $\sigma_{1}, \sigma_{2}, \sigma_{3}, \sigma_{4}$ вполне вещественного поля $\Q(\theta_{1})$, что $\sigma_{1}(\theta_{1}) = \theta_{1}$, $\sigma_{2}(\theta_{1}) = \theta_{2}$, $\sigma_{3}(\theta_{1}) = -\theta_{1}$, $\sigma_{4}(\theta_{1}) = -\theta_{2}$. Пусть $K_{1} = \Q(\theta_{1})$, $K_{2} = \Q(\theta_{2})$, $K_{3} = \Q(-\theta_{1})$, $K_{4} = \Q(-\theta_{2})$ --- сопряженные поля над $\Q$. Ясно, что $K_{1} = K_{3}$, $K_{2} = K_{4}$ и $\sigma^{2}_{3} = \textup{id}$, $\sigma_{4} = \sigma_{2}\sigma_{3}$. 

Предположим, что $K_{1} = K_{2}$. Поскольку $(x - \theta_{1})(x + \theta_{1}) = x^{2}  - 4 - \sqrt{2} \in \Q(\sqrt{2})[x]$, то $K_{1}$ --- нормальное расширение степени $2$ поля $\Q(\sqrt{2})$. При этом $(x - \theta_{2})(x + \theta_{2}) = x^{2}  - 4 +\sqrt{2} \in \Q(\sqrt{2})[x]$. Пусть $\phi \in \gal{\big(K_{1}/\Q(\sqrt{2})\big)}$. Тогда $\phi(\theta_{1}\theta_{2}) = (-\theta_{1})(-\theta_{2}) = \theta_{1}\theta_{2}$, а значит $\theta_{1}\theta_{2} = \sqrt{14}  \in \Q(\sqrt{2})$, чего не может быть. Таким образом, $K_{1} \neq K_{2}$.
 
 Поскольку $\theta = \theta_{1}$ --- примитивный элемент расширения $K_{1}$, то набор чисел $1, \theta, \theta^{2}, \theta^{3}$ является базисом $K_{1}$. Положим $\omega = \theta + \theta^{2}$ и $\psi = -\theta^{2} + \frac{1}{2}\theta^{3}$. Тогда $\omega'' = \sigma_{3}(\omega) = -\theta + \theta^{2}$, $\psi'' = \sigma_{3}(\psi) = -\theta^{2} - \frac{1}{2}\theta^{3}$. Заметим, что $\psi + \psi'' = -2\theta^{2} = - (\omega + \omega'')$ и набор чисел $1$, $\omega$, $\psi$, $\omega''$ является базисом $K_{1}$, поскольку
 
   \[ \begin{pmatrix}
    1 & \phantom{-}0 & \phantom{-}0 & \phantom{-}0 \\
    0 & \phantom{-}1 & \phantom{-}1 & \phantom{-}0 \\
    0 & \phantom{-}0 &                  -1 & \phantom{-}\frac{1}{2} \\
    0 &                  -1 & \phantom{-}1 & \phantom{-}0 
  \end{pmatrix}
    \begin{pmatrix}
      1  \\
      \theta  \\
      \theta^{2} \\
      \theta^{3}
    \end{pmatrix}
    =
    \begin{pmatrix}
      1  \\
      \omega \\
      \psi \\
      \omega''
    \end{pmatrix}.
  \]
 
 Теперь, полагая, что $\alpha = \omega$, $\beta = \psi$, $\gamma = \omega''$, с помощью предложения \ref{prop:more_than_pelle_n_dim} можно построить алгебраическую цепную дробь $\cf(l_1,l_2,l_3,l_4)\in\widetilde{\gA_{3}}$, где подпространство $l_1$ порождено вектором $(1, \alpha, \beta, \gamma)$. Поскольку выполняется утверждение (1) теоремы \ref{theorem_proper}, то $\cf(l_1,l_2,l_3,l_4)$ --- палиндромичная цепная дробь. При этом, поскольку $K_{1} \neq K_{2}$, то $\cf(l_1,l_2,l_3,l_4)$ не обладает циклическими симметриями (см. работу \cite{tlyust}).

\section*{Благодарности}

Автор благодарит О.Н. Германа за внимание к работе и полезные обсуждения результатов.

Автор является победителем конкурса «Junior Leader» Фонда развития теоретической физики и математики «БАЗИС» и хотел бы поблагодарить жюри и спонсоров конкурса.


\begin{thebibliography}{99}
\bibitem
   {klein}  \textsc{F.\,Klein} \textit{Uber eine geometrische Auffassung der gewohnlichen Kettenbruchentwichlung.} Nachr. Ges. Wiss., Gottingen, \textbf{3} (1895), 357--359.
\bibitem
    {korkina_2dim} \textsc{Е.\,И.\,Коркина} \textit{Двумерные цепные дроби. Самые простые примеры.} Тр. МИАН., \textbf{209} (1995), 124--144.
\bibitem
    {german_bordeaux} \textsc{O. \,N. \,German} \textit{Klein polyhedra and lattices with positive norm minima.} Journal de Th\'eorie des Nombres de Bordeaux, \textbf{19} (2007), 157--190.
 \bibitem
 {karpenkov_book} \textsc{O.\,N.\,Karpenkov} \textit{Geometry of Continued Fractions.} Algorithms and Computation in Mathematics, \textbf{26}, Springer-Verlag (2013).
  \bibitem
    {german_tlyust_2} \textsc{О. \,Н. \,Герман, \,И. \,А. \,Тлюстангелов} \textit{Симметрии двумерной цепной дроби.} Изв. РАН. Сер. матем., \textbf{85}:4 (2021), 666--680.
\bibitem
    {german_tlyust} \textsc{O. \,N. \,German, \,I. \,A. \,Tlyustangelov} \textit{Palindromes and periodic continued fractions.} Moscow Journal of Combinatorics and Number Theory, \textbf{6:2-3} (2016), 354--373.
\bibitem
    {galois} \textsc{\'E.\,Galois} \textit{Démonstration d’un théorème sur les fractions continues périodiques.} Annales de Mathématiques, \textbf{19} (1828), 294--301.
\bibitem
    {legendre} \textsc{A.\,M.\,Legendre} \textit{Théorie des nombres.} (3 éd.), Paris (1830).
\bibitem
    {perron_book} \textsc{O.\,Perron} \textit{Die Lehre von den Kettenbr\"uchen. Band I.} (3 Aufl.), Teubner (1954).
\bibitem
    {kraitchik} \textsc{M.\,Kraitchik} \textit{Théorie des nombres. Tome II.} Paris (1926).
  \bibitem
    {tlyust} \textsc{I. \,A. \,Tlyustangelov} \textit{Proper cyclic symmetries of multidimensional continued fractions (in russian).} https://arxiv.org/abs/2111.07192.

	
\

\end{thebibliography}
\end{document}